\documentclass[12pt, reqno]{amsart}
\usepackage{amsmath, amsthm, amscd, amsfonts, amssymb, graphicx, color}
\usepackage[bookmarksnumbered, colorlinks, plainpages]{hyperref}
\hypersetup{colorlinks=true,linkcolor=red, anchorcolor=green, citecolor=cyan, urlcolor=red, filecolor=magenta, pdftoolbar=true}

\newtheorem{theorem}{Theorem}[section]
\newtheorem{lemma}[theorem]{Lemma}

\newtheorem{corollary}[theorem]{Corollary}
\theoremstyle{definition}
\newtheorem{definition}[theorem]{Definition}

\theoremstyle{remark}
\newtheorem{remark}[theorem]{Remark}
\numberwithin{equation}{section}

\begin{document}
\title[Quantum information inequalities]{Quantum information inequalities via tracial positive linear maps}
\author[A. Dadkhah and M.S. Moslehian ]{A. Dadkhah and M. S. Moslehian}
\address{Department of Pure Mathematics, Center Of Excellence in Analysis on Algebraic Structures (CEAAS), Ferdowsi University of Mashhad, P. O. Box 1159, Mashhad 91775, Iran}
\email{dadkhah61@yahoo.com}
\email{moslehian@um.ac.ir}
\subjclass[2010]{46L05, 47A63, 81P15}
\keywords{Tracial positive linear map; trace; generalized covariance, generalized variance; generalized correlation; generalized Wigner--Yanase skew information; $C^*$-algebra}
\begin{abstract}
We present some generalizations of quantum information inequalities involving tracial positive linear maps between $C^*$-algebras. Among several results, we establish a noncommutative Heisenberg uncertainty relation. More precisely, we show that if $\Phi: \mathcal{A} \to \mathcal{B}$ is a tracial positive linear map between $C^*$-algebras , $\rho \in \mathcal{A}$ is a $\Phi$-density element and $A,B$ are self-adjoint operators of $\mathcal{A}$ such that $ {\rm sp}(\mbox{-i}\rho^\frac{1}{2}[A,B]\rho^\frac{1}{2}) \subseteq [m,M] $ for some scalers $0<m<M$, then under some conditions
\begin{eqnarray}\label{inemain1}
V_{\rho,\Phi}(A)\sharp V_{\rho,\Phi}(B)\geq \frac{1}{2\sqrt{K_{m,M}(\rho[A,B])}} \left|\Phi(\rho [A,B])\right|,
\end{eqnarray}
 where $K_{m,M}(\rho[A,B])$ is the Kantorovich constant of the operator $\mbox{-i}\rho^\frac{1}{2}[A,B]\rho^\frac{1}{2}$ and $V_{\rho,\Phi}(X)$ is the generalized variance of $X$.\\ In addition, we use some arguments differing from the scalar theory to present some inequalities related to the generalized correlation and the generalized Wigner--Yanase--Dyson skew information.
\end{abstract} \maketitle
\section{Introduction and preliminaries}
In quantum measurement theory, the classical expectation value of
an observable (self-adjoint operator) $A$ in a quantum state
(density operator) $\rho$ is expressed by ${\rm Tr}(\rho A)$. Also, the classical variance
for a quantum state $\rho$ and an observable operator $A$ is defined by
$ V_\rho(A) := {\rm Tr}(\rho A^2) -({\rm Tr}(\rho A))^2$. The
Heisenberg uncertainty relation asserts that
\begin{eqnarray}\label{heis1}
V_{\rho}(A)V_{\rho}(B) \geq \frac{1}{4} |{\rm Tr}(\rho [A,B])|^2
\end{eqnarray}
for a quantum state $\rho$ and two observables $A$ and $B$; see \cite{He1}. It gives a fundamental limit for the
measurements of incompatible observables. A further
strong result was given by Schr\"{o}dinger \cite{schro} as
\begin{eqnarray}\label{heis2}
V_{\rho}(A)V_{\rho}(B) -|{\rm Re} ({\rm Cov}_{\rho}(A,B))|^2\geq \frac{1}{4} |{\rm Tr}(\rho [A,B])|^2,
\end{eqnarray}
where $[A,B]:=AB-BA$ is the commutator of $A,B$ and the classical covariance is defined by ${\rm Cov_\rho}(A) := {\rm Tr}(\rho AB) -{\rm Tr}(\rho A) {\rm Tr}(\rho B)$.

Yanagi et al. \cite{yangi} defined the one-parameter correlation and the one-parameter Wigner--Yanase skew information (is known as the Wigner--Yanase--Dyson skew information; cf. \cite{lib}) for operators $A,B$, respectively, as follows
\begin{eqnarray*}
{\rm Corr}_{\rho}^{\alpha}(A,B):={\rm Tr}(\rho A^*B)-{\rm Tr}(\rho^{1-\alpha} A^*\rho^{ \alpha}B) \mbox{\quad and \quad} I_{\rho}^{\alpha}(A):={\rm Corr}_{\rho}^{\alpha}(A,A),
\end{eqnarray*}
where $\alpha \in [0,1]$. They showed a trace inequality representing the relation between these two quantities as
\begin{eqnarray}\label{semic}
\left| {\rm Re(Corr}_\rho^{\alpha} (A,B))\right|^2 \leq I_{\rho}^\alpha (A)I_{\rho}^\alpha (B).
\end{eqnarray}
In the case that $\alpha =\frac{1}{2}$, we get the classical notions of the correlation ${\rm Corr}_{\rho}(A,B)$ and the Wigner--Yanase skew information $I_{\rho}(A)$. The classical Wigner--Yanase skew information represents a for non-commutativity between a quantum state $\rho$ and an observable $A$.\\
Luo \cite{luo} introduced the quantity $U_\rho(A)$ as a measure of uncertainty by
\begin{eqnarray*}
U_\rho(A)=\sqrt{V_\rho(A)^2-(V_\rho(A)-I_\rho(A))^2}.
\end{eqnarray*}
He then showed a Heisenberg-type uncertainty relation on $U_{\rho}(A)$ as
\begin{eqnarray}\label{skhies}
U_{\rho}(A)U_{\rho}(B)\geq \frac{1}{4} |{\rm Tr}(\rho [A,B])|^2.
\end{eqnarray}
These inequalities was studied and extended by a number of mathematicians. For further information we refer interested readers to \cite{fried, gib, yoo, petz}.

Let $\mathbb{B}(\mathcal{H})$ denote the $C^*$-algebra of all bounded linear operators on a complex Hilbert space $(\mathcal H, \langle\cdot,\cdot\rangle)$ with the unit $I$. If $\mathcal{H}=\mathbb{C}^n$, we identify $\mathbb{B}(\mathbb{C}^n)$ with the matrix algebra of $n\times n$ complex matrices $M_n(\mathbb{C})$. We consider the usual L\"{o}wner order $\leq$ on the real space of self-adjoint operators. Throughout the paper, a capital letter means an operator in $\mathbb{B}(\mathcal{H})$. An operator A is said to be strictly positive (denoted by $ A >0$) if it is a positive invertible operator. According to the Gelfand--Naimark--Segal theorem, every $C^*$-algebra can be regarded as a $C^*$-subalgebra of $\mathbb{B}(\mathcal{H})$ for some Hilbert space $\mathcal{H}$. We use $\mathcal{A},\mathcal{B}, \cdots $ to denote $C^*$-algebras. We denote by ${\rm Re}(A)$ and ${\rm Im}(A)$ the real and imaginary parts of $A$, respectively, so we may consider elements of $\mathcal{A}$ as Hilbert space operators. The geometric mean is defined by $A\sharp B=A^\frac{1}{2}\left(A^{-\frac{1}{2}}BA^{-\frac{1}{2}}\right)^\frac{1}{2}A^\frac{1}{2}$ for operators $A>0$ and $B\geq 0$. A $W^*$-algebra is a $*$-algebra of bounded operators on a Hilbert space that is closed in the weak operator topology and contains the identity operator. The $C^*$-algebra of complex valued continuous functions on the compact Hausdorff space $\Omega$ is denoted by $C(\Omega)$.\\
A linear map $\Phi :\mathcal{A}\longrightarrow \mathcal{B}$ between $C^*$-algebras is said to be $*$-linear if $\Phi(A^*)=\Phi(A)^*$. It is positive if $\Phi(A)\geq 0$ whenever $A \geq 0$. It is called strictly positive if $A>0$, then $\Phi(A)>0$.
We say that $\Phi$ is unital if $\mathcal{A}, \mathcal{B}$ are unital and $\Phi$ preserves the unit. A linear map $\Phi$ is called $n$-positive if the map $\Phi_n : M_n(\mathcal{A})\longrightarrow M_n(\mathcal{B})$ defined by $\Phi_n([a_{ij}]) = [\Phi(a_{ij})]$ is positive, where $M_n(\mathcal{A})$ stands for the $C^*$-algebra of
$n \times n$ matrices with entries in $\mathcal{A}$. A map $\Phi$ is said to be completely positive if it is $n$-positive for every $n\in \mathbb{N}$. According to \cite[Theorem 1.2.4]{stormer} if the range of the positive linear map $\Phi$ is commutative, then $\Phi$ is completely positive. It is known (see, e.g., \cite{mond}) that if $\Phi$ is a unital positive linear map, then
\begin{eqnarray}\label{meanphi}
\Phi(A\sharp B)\leq \Phi(A)\sharp \Phi(B).
\end{eqnarray}

A map $\Phi$ is called tracial if $\Phi(AB)=\Phi(BA)$. The usual trace on the trace class operators acting on a Hilbert space is a tracial positive linear functional. For a given closed two sided ideal $\mathcal{I}$ of a $C^*$-algebra $\mathcal{A}$, the existence of a tracial positive linear map $\Phi: \mathcal{A}\to \mathcal{A}$ satisfying $\Phi(\Phi(A))=\Phi(A)$ and $\Phi(A)-A\in\mathcal{I}$ is equivalent to the commutativity of the quotient $\mathcal{A}/\mathcal{I}$; see \cite{choi} for more examples and implications of the definition. For a tracial positive linear map $\Phi$, a positive operator $\rho \in \mathcal{A}$ is said to be $\Phi$-density if $\Phi(\rho)=I$. A unital $C^*$-algebra $\mathcal{B}$ is said to be injective whenever for every unital $C^*$-algebra $\mathcal{A}$ and for every self-adjoint subspace $S$ of $\mathcal{A}$, each unital completely positive linear map from $S$ into $\mathcal{B}$, can be extended to a completely positive linear map from $\mathcal{A}$ into $\mathcal{B}$.
Our investigation is based on the following definition.
\begin{definition}\label{def}
Let $\Phi:\mathcal{A}\longrightarrow \mathcal{B}$ be a tracial positive linear map and $\rho$ be a $\Phi$-density operator. Then
\begin{eqnarray*}
{\rm Cov_{\rho,\Phi}}(A,B):=\Phi(\rho A^*B)-\Phi(\rho A^*)\Phi(\rho B)\ {\rm and} \ V_{\rho,\Phi}(A):={\rm Cov}_{\rho,\Phi}(A,A),
\end{eqnarray*}
are called the generalized covariance and the generalized variance $A,B$, respectively. Further, the generalized correlation and the generalized Wigner--Yanase--Dyson skew information of two operators $A,B$ are defined by
\begin{eqnarray*}
{\rm Corr}_{\rho,\Phi}^{\alpha}(A,B):=\Phi(\rho A^*B)-\Phi(\rho^{ 1-\alpha} A^*\rho^\alpha B) \mbox{\quad and \quad} I_{\rho,\Phi}^{\alpha}(A):={\rm Corr}_{\rho,\Phi}^{\alpha}(A,A)\,,
\end{eqnarray*}
respectively.
\end{definition}
It is known that for every tracial positive linear map, the matrix
\begin{eqnarray*}
 \begin{bmatrix} V_{\rho,\Phi}(A) & {\rm Cov}_{\rho,\Phi}(B,A)\\ {\rm Cov}_{\rho,\Phi}(A,B) & V_{\rho,\Phi}(B)\end{bmatrix}
\end{eqnarray*}
is positive, which is equivalent to
\begin{eqnarray}
V_{\rho,\Phi}(A)\geq {\rm Cov}_{\rho,\Phi}(B,A) (V_{\rho,\Phi}(B))^{-1} {\rm Cov}_{\rho,\Phi}(A,B),
\end{eqnarray}
which is called the variance-covariance inequality; see \cite{mo1, MM} for technical discussions.

If $\mathcal{A}$ is a $C^*$-algebra and $\mathcal{B}$ is a $C^*$-subalgebra of $\mathcal{A}$, then a conditional expectation $\mathcal{E} : \mathcal{A}\longrightarrow \mathcal{B}$ is a positive contractive linear map such that $\mathcal{E}(BAC)=B\mathcal{E}(A)C$ for every $A\in\mathcal{A}$ and all $B,C \in \mathcal{B}$.\\
If $(\mathcal{X},\langle \cdot ,\cdot \rangle)$ is a semi-inner product module over a $C^*$-algebra $\mathcal{A}$, then
the Cauchy--Schwarz inequality for $x,y\in \mathcal{X}$ asserts that (see \cite{lanc, ABFM})
 \begin{eqnarray*}
 \langle x,y\rangle \langle y,x\rangle\leq \|\langle y,y\rangle \| \langle x,x\rangle .
\end{eqnarray*}
 If $\langle y,y\rangle \in \mathcal{Z}(\mathcal{A})$, where $\mathcal{Z}(\mathcal{A})$ is the center of the $C^*$-algebra $\mathcal{A}$, then the latter inequality turns into (see \cite{il1})
 \begin{eqnarray}\label{cuachysharp}
 \langle x,y\rangle \langle y,x\rangle\leq \langle y,y\rangle \langle x,x\rangle .
\end{eqnarray}

In Section 1, we use some techniques in the non-commutative setting to give some generalizations of inequalities (\ref{heis2}) and (\ref{heis1}) for tracial positive linear maps between $C^*$-algebras. More precisely, for a tracial positive linear map $\Phi$ between $C^*$-algebras under ceratin conditions we show that
\begin{eqnarray*}
V_{\rho,\Phi}(A)\sharp V_{\rho,\Phi}(B)\geq \frac{1}{2\sqrt{K_{m,M}(\rho[A,B])}} \left|\Phi(\rho [A,B])\right|
\end{eqnarray*}
for every self adjoint operators $A,B$. Section 2 deals with generalizations of inequalities (\ref{semic}) and (\ref{skhies}) for conditional expectation maps. Among other things, we prove that
\begin{eqnarray*}
0\leq I_{\rho,\Phi}^{\alpha}(A)\leq I_{\rho,\Phi}(A) \leq V_{\rho,\Phi}(A).
\end{eqnarray*}
 for every self-adjoint operator $A$. In addition, we generalize some significant inequalities for trace in the quantum mechanical systems to inequalities for tracial positive linear maps between $C^*$ -algebras. We indeed use some arguments differing from the classical theory to present some inequalities related to the generalized correlation and the generalized Wigner--Yanase--Dyson skew information.

\section{Inequalities for generalized covariance and variance}

We start this section by giving a generalization of inequality (\ref{heis2}). In fact we prove inequality (\ref{heis2}) for a tracial positive linear map between $C^*$-algebras under some mild conditions. We need the following notions slightly differing from the notions defined in Definition \ref{def}.
\begin{definition}
For a tracial positive linear map $\Phi$ from a $C^*$-algebra $\mathcal{A}$ into a $C^*$-algebra $\mathcal{B}$ and positive operator $\rho \in \mathcal{A}$ and for operators $ A,B \in \mathcal{A}$ we set
\begin{eqnarray*}
{\rm Cov_{\rho,\Phi}'}(A,B)=\Phi(\rho A^*B)-\Phi(\rho A^*)\Phi(\rho)^{-1}\Phi(\rho B)\ {\rm and} \ V_{\rho,\Phi}'(A)=\rm Cov_{\rho,\Phi}'(A,A).
\end{eqnarray*}
\end{definition}

To achieve our result we need the following lemma.

\begin{lemma}\label{matrixp}\cite[Lemma 2.1]{choi}
Let $A>0, B\geq 0$ be two operators in $\mathcal{A}$. Then the block matrix $\begin{bmatrix}A& X\\ X^*& B \end{bmatrix}$ is positive if and only if $B\geq X^* A^{-1} X$.
\end{lemma}

We are ready to prove our first result.

\begin{theorem}\label{heiscom}
Let $\Phi :\mathcal{A} \longrightarrow \mathcal{B}$ be a tracial positive linear map between $C^*$-algebras and $\rho \in \mathcal{A}$ be a positive operator such that $\Phi(\rho)>0$. If $\Phi(\mathcal{A})$ is a commutative subspace of $\mathcal{B}$, then
\begin{eqnarray*}
V_{\rho,\Phi}'(A)V_{\rho,\Phi}'(B) -|{\rm Re} ({\rm Cov}_{\rho,\Phi}'(A,B))|^2\geq \frac{1}{4} |\Phi(\rho [A,B])|^2
\end{eqnarray*}
for all self-adjoint operators $A,B$.
In particular,
\begin{eqnarray*}
V_{\rho,\Phi}'(A)V_{\rho,\Phi}'(B) \geq \frac{1}{4} |\Phi(\rho [A,B])|^2.
\end{eqnarray*}
\end{theorem}
\begin{proof}
A simple calculation shows that
\begin{eqnarray*}
{\rm Cov}_{\rho,\Phi}'(A,B)-{\rm Cov}_{\rho,\Phi}'(B,A)&=& \Phi(\rho AB)- \Phi(\rho A)\Phi(\rho)^{-1}\Phi(\rho B)\\ && - \Phi(\rho BA)+\Phi(\rho B)\Phi(\rho)^{-1}\Phi(\rho A)\\ &=&\Phi(\rho[A,B]) \ \ ({\rm since,\ \Phi(\mathcal{A})\ is\ commutative})
\end{eqnarray*}
and
\begin{eqnarray*}
{\rm Cov}_{\rho,\Psi}'(A,B)+{\rm Cov}_{\rho,\Phi}'(B,A)&=&{\rm Cov}_{\rho,\Psi}'(A,B)+\left({\rm Cov}_{\rho,\Phi}'(A,B)\right)^*\\ &=&2{\rm Re}({\rm Cov}_{\rho,\Phi}'(A,B)).
\end{eqnarray*}
Summing both sides of the above inequalities, we get
\begin{eqnarray*}
2{\rm Cov}_{\rho,\Phi}'(A,B) =\Phi(\rho[A,B])+2{\rm Re}({\rm Cov}_{\rho,\Phi}'(A,B)).
\end{eqnarray*}
Since $\Phi(\rho[A,B])^* =-\Phi(\rho[A,B])$ and $2{\rm Re}({\rm Cov}_{\rho\Phi}'(A,B))$ is self-adjoint, it follows from the commutativity of $\Phi(\mathcal{A})$ that
\begin{eqnarray}\label{rabete}
|{\rm Cov}_{\rho,\Phi}'(A,B)|^2=|{\rm Re} ({\rm Cov}_{\rho,\Phi}'(A,B))|^2+ \frac{1}{4} |\Phi(\rho [A,B])|^2.
\end{eqnarray}
Moreover,
\begin{eqnarray*}
\begin{bmatrix} \rho^\frac{1}{2} A^*A\rho^\frac{1}{2} & \rho^\frac{1}{2}A^*B \rho^\frac{1}{2}& \rho^\frac{1}{2} A^* \rho^\frac{1}{2} \\ \rho^\frac{1}{2}B^*A \rho^\frac{1}{2} & \rho^\frac{1}{2}B^*B \rho^\frac{1}{2} & \rho^\frac{1}{2}B^* \rho^\frac{1}{2}\\ \rho^\frac{1}{2}A \rho^\frac{1}{2} & \rho^\frac{1}{2}B \rho^\frac{1}{2}& \rho
\end{bmatrix}= \begin{bmatrix} \rho^\frac{1}{2}A^*&0&0 \\ \rho^\frac{1}{2}B^*&0&0 \\ \rho^\frac{1}{2} &0&0 \end{bmatrix}\begin{bmatrix} A \rho^\frac{1}{2}&B \rho^\frac{1}{2}& \rho^\frac{1}{2}\\ 0&0&0 \\ 0&0&0 \end{bmatrix}\geq 0.
\end{eqnarray*}
It follows from the complete positivity (and indeed the 3-positivity) and the tracial property of $\Phi$ that
 \begin{eqnarray*}
\begin{bmatrix} \Phi(\rho A^*A) & \Phi(\rho A^*B) & \Phi(\rho A^*) \\ \Phi(\rho B^*A) &\Phi(\rho B^*B)& \Phi(\rho B^*) \\ \Phi(\rho A)& \Phi(\rho B) & \Phi(\rho)
\end{bmatrix} \geq 0.
\end{eqnarray*}
Hence, by applying Lemma \ref{matrixp}, we have
\begin{eqnarray*}
\begin{bmatrix} \Phi(\rho A^*A) & \Phi(\rho A^*B) \\ \Phi(\rho B^*A) & \Phi(\rho B^*B)
\end{bmatrix} \geq \begin{bmatrix} \Phi(\rho A)^* & 0 \\ \Phi(\rho B)^* & 0\end{bmatrix} \begin{bmatrix} \Phi(\rho)^{-1} & 0\\ 0 & 0
\end{bmatrix} \begin{bmatrix} \Phi(\rho A) & \Phi(\rho B) \\ 0 & 0\end{bmatrix},
\end{eqnarray*}
whence
\begin{eqnarray*}
\begin{bmatrix} \Phi(\rho A^*A) & \Phi(\rho A^*B) \\ \Phi(\rho B^*A) & \Phi(\rho B^*B)
\end{bmatrix} \geq \begin{bmatrix} \Phi(\rho A)^* \Phi(\rho)^{-1}\Phi(\rho A) & \Phi(\rho A)^* \Phi(\rho)^{-1}\Phi(\rho B) \\ \Phi(\rho B)^* \Phi(\rho)^{-1}\Phi(\rho A) & \Phi(\rho B)^* \Phi(\rho)^{-1}\Phi(\rho B)\end{bmatrix},
\end{eqnarray*}
or equivalently ,
\begin{eqnarray*}
 \begin{bmatrix} V_{\rho,\Phi}'(A) & {\rm Cov}_{\rho,\Phi}'(B,A)\\ {\rm Cov}_{\rho,\Phi}'(A,B) & V_{\rho,\Phi}'(B)\end{bmatrix}\geq 0.
\end{eqnarray*}
It follows from Lemma \ref{matrixp} that
\begin{eqnarray*}
 V_{\rho,\Phi}'(A)\geq {\rm Cov}_{\rho,\Phi}'(A,B)^* \left( V_{\rho,\Phi}'(B) \right)^{-1} {\rm Cov}_{\rho,\Phi}'(A,B).
\end{eqnarray*}
Applying the commutativity $\Phi(\mathcal{A})$, we get
\begin{eqnarray*}
 V_{\rho,\Phi}'(A) V_{\rho,\Phi}'(B)\geq|{\rm Cov}_{\rho,\Phi}'(A,B)|^2.
\end{eqnarray*}
Consequently, if $A,B$ are self-adjoint operators, then
\begin{align*}
\big(\Phi(\rho A^2)-\Phi(\rho A)^2\Phi(\rho)^{-1}\big)&\big( \Phi(\rho B^2)-\Phi(\rho B)^2\Phi(\rho)^{-1}\big) \\ &=V_{\rho,\Phi}'(A) V_{\rho,\Phi}'(B)\\ &\geq |{\rm Re} ({\rm Cov}_{\rho,\Psi}'(A,B))|^2+ \frac{1}{4} |\Phi(\rho [A,B])|^2 \\ & \qquad \qquad \qquad ({\rm by \ equality \ (\ref{rabete}})).
\end{align*}
\end{proof}

\begin{corollary}\label{heiscomcor}
Let $\Phi :\mathcal{A} \longrightarrow \mathcal{B}$ be a tracial positive linear map between $C^*$-algebras and $\rho \in \mathcal{A}$ be a $\Phi$-density operator. If $\Phi(\mathcal{A})$ is a commutative subspace of $\mathcal{B}$, then
\begin{eqnarray*}
V_{\rho,\Phi}(A)V_{\rho,\Phi}(B) -|{\rm Re} ({\rm Cov}_{\rho,\Phi}(A,B))|^2\geq \frac{1}{4} |\Phi(\rho [A,B])|^2
\end{eqnarray*}
for all self-adjoint operators $A,B$.
\end{corollary}
\begin{proof}
Obviously, if $\rho$ is a $\Phi$-density operator, then $\Phi(\rho)^{-1}=I$. Now Theorem \ref{heiscom} yields the required inequality.
\end{proof}

Let $\mathcal{A}$ be a $C^*$-algebra and $\mathcal{B}$ be a $C^*$-subalgebra of $\mathcal{A}$. If $\mathcal{E}:\mathcal{A}\longrightarrow \mathcal{B}$ is a tracial conditional expectation, then
\begin{eqnarray}\label{concom}
B \mathcal{E}(A) = \mathcal{E}(BA)=\mathcal{E}(AB)= \mathcal{E}(A)B
\end{eqnarray}
for every $A\in \mathcal{A}$ and $B\in \mathcal{B}$. Using this fact we give the following corollary.
\begin{corollary}
Let $\mathcal{A}$ be a $C^*$-algebra and $\mathcal{B}$ be a $C^*$-subalgebra of $\mathcal{A}$. If $\mathcal{E}:\mathcal{A}\longrightarrow \mathcal{B}$ is a tracial conditional expectation, then
\begin{eqnarray*}
V_{\rho,\mathcal{E}}(A)V_{\rho,\mathcal{E}}(B) -|{\rm Re} ({\rm Cov}_{\rho,\mathcal{E}}(A,B))|^2\geq \frac{1}{4} |\mathcal{E}(\rho [A,B])|^2
\end{eqnarray*}
for all self-adjoint operators $A,B \in \mathcal{A}$ and each $\mathcal{E}$-density operator $\rho \in \mathcal{A}$.
\end{corollary}
Now we give a version of Heisenberg's uncertainty relation, in the case that $\mathcal{B}$ is not a commutative $C^*$-algebra. To get this result we need some lemmas.

\begin{lemma}[Choi--Tsui]\label{choi} \cite[pp. 59-60]{choi}
Let $\mathcal{A},\mathcal{B}$ be $C^*$-algebras such that either one of them is $W^*$-algebra or $\mathcal{B}$ is an injective $C^*$-algebra. Let $\Phi:\mathcal{A}\longrightarrow \mathcal{B}$ be a tracial positive linear map. Then there exist a
commutative $C^*$-algebra $C(X)$ and tracial positive linear maps $\phi_1 : A\longrightarrow C(X) $ and $\phi_2: C(X)\longrightarrow \mathcal{B}$ such that $\Phi=\phi_2\circ \phi_1$. Moreover, in the case that $\Phi$ is unital, then
 $\phi_1$ and $\phi_2$ can be chosen to be unital. In particular, $\Phi$ is completely positive.
\end{lemma}

\begin{lemma}\label{revkadison}{\rm(Kadison's inequality)} \cite[Chapter 1]{mond} If $\Phi:\mathcal{A}\longrightarrow \mathcal{B}$ is a unital $2$-positive linear map between unital $C^*$-algebras, then
\begin{eqnarray*}
\Phi(|A|^2)\geq |\Phi(A)|^2
\end{eqnarray*}

for every $A\in \mathcal{A}$.
\end{lemma}
In the case that $A$ is a positive operator of $\mathcal{A}$ satisfying $0<mI\leq A\leq MI$ for some scalers $m<M$, by \cite[Theorem 1.32]{mond}, the reverse inequality
\begin{eqnarray}\label{revk}
\Phi(A^2)\leq \dfrac{(M+m)^2}{4Mm}\Phi(A)^2
\end{eqnarray}
holds.

\begin{lemma}
Let $\Phi:\mathcal{A}\longrightarrow \mathcal{B}$ be a unital $2$-positive linear map between unital $C^*$-algebras. If $A$ is an operator of $\mathcal{A}$ satisfying ${\rm sp}(A) \subseteq [m,M]\cup [-M,-m]$ for some scalers $0<m<M$, then
\begin{eqnarray}\label{revkadison1}
|\Phi(A)|\leq \sqrt{\dfrac{(M+m)^2}{4Mm}}\Phi(|A|).
\end{eqnarray}
\end{lemma}

\begin{proof}
Using Lemma \ref{revkadison} and inequality (\ref{revk}), we get
\begin{eqnarray*}
|\Phi(A)|^2\leq \Phi(|A|^2) \leq \dfrac{(M+m)^2}{4Mm}(\Phi(|A|))^2.
\end{eqnarray*}
Taking square roots of both sides of the latter inequality we obtain inequality (\ref{revkadison1}).
\end{proof}

In this paper, we denote $\dfrac{(M+m)^2}{4Mm}$ for an operator $m \leq A \leq M$ by $K_{M,m}(A)$, which is called the Kantorovich constant of $A$.\\
The next theorem gives a Heisenberg's type uncertainty relation for tracial positive linear maps between $C^*$-algebras.
\begin{theorem}\label{main1}
Let $\mathcal{A},\mathcal{B}$ be unital $C^*$-algebras and $\Omega$ be a compact Hausdorff space. Let $\phi_1 :\mathcal{A} \longrightarrow C(\Omega) $ be a unital tracial positive linear map and $\phi_2 :C(\Omega) \longrightarrow \mathcal{B}$ be a unital positive linear map and $\Phi:=\phi_2 \circ \phi_1$. If $\rho \in \mathcal{A}$ is a $\Phi$-density operator and $A,B$ are self-adjoint operators of $\mathcal{A}$ such that $ {\rm sp}(\mbox{-i}\rho^\frac{1}{2}[A,B]\rho^\frac{1}{2}) \subseteq [m,M] $ for some scalers $0<m<M$, then
\begin{eqnarray}\label{inemain1}
V_{\rho,\Phi}(A)\sharp V_{\rho,\Phi}(B)\geq \frac{1}{2\sqrt{K_{m,M}(\rho[A,B])}} \left|\Phi(\rho [A,B])\right|,
\end{eqnarray}
 where $K_{m,M}(\rho[A,B])$ is the Kantorovich constant of the operator $\mbox{-i}\rho^\frac{1}{2}[A,B]\rho^\frac{1}{2}$.
\end{theorem}
\begin{proof}
By a continuity argument we can assume that $\rho>0$. Due to $0<mI\leq \mbox{-i} \rho^\frac{1}{2}[A,B]\rho^\frac{1}{2} \leq MI$ and $\phi_1$ is unital and tracial positive linear, we infer that $mI\leq \mbox{-i}\phi_1( \rho[A,B] )\leq MI$. It follows that
 \begin{eqnarray}
 mI\leq |\phi_1( \rho[A,B] )|\leq MI.
 \end{eqnarray}
Using the fact that $\phi_1$ is a unital positive linear map (and so strictly positive) and applying Theorem \ref{heiscom} for $ \phi_1$ we get
 \begin{eqnarray*}
\left(\phi_1(\rho A^2)-\phi_1(\rho A)^2\phi_1(\rho)^{-1}\right) \left(\phi_1(\rho B^2)-\phi_1(\rho B)^2\phi_1(\rho)^{-1}\right) \geq \frac{1}{4} |\phi_1(\rho [A,B])|^2.
 \end{eqnarray*}
Applying the commutativity of range of $\phi_1$ we get
\begin{eqnarray}\label{varcomm}
\left(\phi_1(\rho A^2)-\phi_1(\rho A)^2\phi_1(\rho)^{-1}\right) \sharp \left(\phi_1(\rho B^2)-\phi_1(\rho B)^2\phi_1(\rho)^{-1}\right) \geq \frac{1}{2} |\phi_1(\rho [A,B])|.
\end{eqnarray}
Using the fact that $(\phi_2\circ\phi_1)(\rho)=I$, we can write
\begin{align*}
 \left(\Phi(\rho A^2)- \Phi (\rho A)^2\right) &\sharp \left(\Phi(\rho B^2)-\Phi(\rho B)^2\right)\\ &= \left((\phi_2 \circ\phi_1)(\rho A^2)- (\phi_2\circ\phi_1)(\rho A)^2\right) \\ &\quad\sharp \left((\phi_2\circ\phi_1)(\rho B^2)-(\phi_2\circ \phi_1)(\rho B)^2\right)\\&= \left((\phi_2 \circ\phi_1)(\rho A^2)- (\phi_2\circ\phi_1)(\rho A) (\phi_2\circ\phi_1)(\rho) (\phi_2\circ\phi_1)(\rho A)\right)\\ &\quad \sharp \left((\phi_2\circ\phi_1)(\rho B^2)-(\phi_2\circ \phi_1)(\rho B) (\phi_2\circ \phi_1)(\rho)(\phi_2\circ \phi_1)(\rho B)\right).
 \end{align*}
We claim that
\begin{align*}
\big((\phi_2 \circ\phi_1)&(\rho A^2)-(\phi_2\circ\phi_1)(\rho A) (\phi_2\circ\phi_1)(\rho)(\phi_2\circ\phi_1)(\rho A)\big)\\ &\quad\sharp \left((\phi_2\circ\phi_1)(\rho B^2)-(\phi_2\circ \phi_1)(\rho B) (\phi_2\circ \phi_1)(\rho)(\phi_2\circ \phi_1)(\rho B)\right)
\\ &\geq \phi_2 \big(\phi_1(\rho A^2)-\phi_1(\rho A)^2 \phi_1(\rho)^{-1}\big) \sharp \phi_2 \big(\phi_1(\rho B^2)-\phi_1(\rho B)^2 \phi_1(\rho)^{-1}\big).
\end{align*}
Since the range of $\phi_1$ is commutative, to prove our claim, it is enough to show that
\begin{eqnarray}\label{rah}
\phi_2\left(\phi_1(\rho X) \phi_1(\rho)^{-1}\phi_1(\rho X)\right) \geq (\phi_2\circ\phi_1)(\rho X) (\phi_2\circ\phi_1)(\rho)(\phi_2\circ\phi_1)(\rho X)
\end{eqnarray}
for every self-adjoint operator $X\in \mathcal{A}$ and $(\phi_2\circ\phi_1)$-density operator $\rho\in \mathcal{A}$.
Clearly, matrix $\begin{pmatrix} \phi_1(\rho X) \phi_1(\rho)^{-1}\phi_1(\rho X) & \phi_1(\rho X)\\ \phi_1(\rho X) & \phi_1(\rho) \end{pmatrix}$ is positive. Since $\Phi_2$ is completely positive (and so 2-positive), we get
\begin{eqnarray*}
\begin{pmatrix} \phi_2(\phi_1(\rho X) \phi_1(\rho)^{-1}\phi_1(\rho X)) & (\phi_2\circ \phi_1)(\rho X)\\ (\phi_2\circ \phi_1)(\rho X) & (\phi_2\circ \phi_1)(\rho) \end{pmatrix} \geq 0,
\end{eqnarray*}
which ensures the validity of inequality (\ref{rah}).
Therefore,
 \begin{eqnarray}
 \nonumber V_{\rho,\Phi}(A)\sharp V_{\rho,\Phi}(B) &=& \left(\Phi(\rho A^2)- \Phi (\rho A)^2\right) \sharp \left(\Phi(\rho B^2)-\Phi(\rho B)^2\right)\\ \nonumber &=& \left((\phi_2 \circ\phi_1)(\rho A^2)- (\phi_2\circ\phi_1)(\rho A)^2\right) \\ \nonumber &&\sharp \left((\phi_2\circ\phi_1)(\rho B^2)-(\phi_2\circ \phi_1)(\rho B)^2\right)\\ \nonumber &=& \left((\phi_2 \circ\phi_1)(\rho A^2)- (\phi_2\circ\phi_1)(\rho A) (\phi_2\circ\phi_1)(\rho) (\phi_2\circ\phi_1)(\rho A)\right)\\ \nonumber && \sharp \left((\phi_2\circ\phi_1)(\rho B^2)-(\phi_2\circ \phi_1)(\rho B) (\phi_2\circ \phi_1)(\rho)(\phi_2\circ \phi_1)(\rho B)\right) \\ \nonumber &\geq& \phi_2 \big(\phi_1(\rho A^2)-\phi_1(\rho A)^2 \phi_1(\rho)^{-1}\big) \\ \nonumber && \sharp \phi_2 \big(\phi_1(\rho B^2)-\phi_1(\rho B)^2 \phi_1(\rho)^{-1}\big)\\ \nonumber &&\qquad \qquad \qquad \qquad \qquad \ \ \ \ \ \ \ \ \ \ ({\rm by \ inequality\ (\ref{rah})}) \\ \nonumber &\geq& \phi_2 \Big( \big(\phi_1(\rho A^2)-\phi_1(\rho A)^2 \phi_1(\rho)^{-1}\big) \\ \nonumber && \sharp \big(\phi_1(\rho B^2)-\phi_1(\rho B)^2 \phi_1(\rho)^{-1}\big)\Big) \\ \nonumber &&\qquad \qquad \qquad \qquad \qquad \ \ \ \ \ \ \ \ \ \ ({\rm by \ inequality (\ref{meanphi}) })\\ \label{namosavi}
 &\geq & \frac{1}{2} \phi_2(|\phi_1(\rho [A,B])|)\\ \nonumber &&\qquad \qquad \qquad \qquad \qquad \ \ \ \ \ \ \ \ \ \ ({\rm by\ inequality\ (\ref{varcomm}})\\ \nonumber
 &\geq & \frac{1}{2\sqrt{K_{m,M}(\rho[A,B])}} \left|\phi_2\circ\phi_1(\rho [A,B])\right|\\ \nonumber &&\qquad \qquad \qquad \qquad \qquad \ \ \ \ \ \ \ \ \ \ ({\rm by\ inequality \ (\ref{revkadison1})})\\ \nonumber &=&\frac{1}{2\sqrt{K_{m,M}(\rho[A,B])}} |\Phi(\rho [A,B])|.
\end{eqnarray}
\end{proof}
Applying Lemma \ref{choi}, we immediately get the following corollary.
\begin{corollary}
Let $\mathcal{A},\mathcal{B}$ be $C^*$-algebras such that either one of them is $W^*$-algebra or $\mathcal{B}$ is an injective $C^*$-algebra. Let $\Phi:\mathcal{A}\longrightarrow \mathcal{B}$ be a tracial positive linear map and $\rho \in \mathcal{A}$ be a $\Phi$-density operator and $A,B$ are self-adjoint operators in $ \mathcal{A}$ such that $ {\rm sp}(\mbox{-i}\rho^\frac{1}{2}[A,B]\rho^\frac{1}{2})\subseteq [m,M] $ for some scalars $0<m<M$, then
\begin{eqnarray*}
V_{\rho,\Phi}(A)\sharp V_{\rho,\Phi}(B)\geq \frac{1}{2\sqrt{K_{m,M}(\rho[A,B])}} \left|\Phi(\rho [A,B])\right|,
\end{eqnarray*}
 where $K_{[m,M]}(\rho[A,B])$ is the Kantorovich constant of the operator $\mbox{-i}\rho^\frac{1}{2}[A,B]\rho^\frac{1}{2}$.
\end{corollary}
If $\phi_1:={\rm Tr}$ is the usual trace and $\phi_2$ is the identity map on $\mathbb{C}$, then it immediately follows from inequality (\ref{namosavi}) that
\begin{eqnarray*}
V_{\rho}(A)^\frac{1}{2}V_{\rho}(B)^\frac{1}{2 }\geq \frac{1}{2} \left|{\rm Tr}(\rho [A,B])\right|.
\end{eqnarray*}
So we get the following result.
\begin{corollary}
For every self-adjoint operators $A,B$ and each density operator $\rho$ it holds that
\begin{eqnarray*}
V_{\rho}(A)V_{\rho}(B)\geq \frac{1}{4} \left|{\rm Tr}(\rho [A,B])\right|^2.
\end{eqnarray*}
\end{corollary}
\begin{remark}
Let $\mathcal{A}$ and $\mathcal{B}$ be  unital $C^*$-algebras. By readout the proof of Theorem \ref{heiscom}, Theorem \ref{main1} and Corollary \ref{heiscomcor}, if we put $\rho:=I$, then we can replace the condition ``tracial positive linear map'' with ``unital positive linear map''.
\end{remark}

\section{Inequalities for generalized correlation and Wigner--Yanase skew information}
We aim to give generalizations of inequality (\ref{semic}) and inequality (\ref{skhies}). In addition we generalize some inequalities related to classical Wigner--Yanase--Dyson skew information to tracial positive linear map. As mentioned in the introduction, in the case that $\alpha =\frac{1}{2}$ we denote $I_{\rho,\Phi}^{\alpha}$ by $I_{\rho,\Phi}$. In some cases we prove our result by assuming that $\rho$ is only a positive operator. To get these generalizations we need some lemmas. The technique of the first lemma is classic.
\begin{lemma}
Let $\mathcal{A}$ and $\mathcal{B}$ be $C^*$-algebras. If $\Phi :\mathcal{A} \longrightarrow \mathcal{B}$ is a tracial positive linear map, then
\begin{eqnarray}\label{skew}
I_{\rho,\Phi}^{\alpha}(A)= \Phi(\rho A^2) -\Phi(\rho^{1-\alpha}A\rho^\alpha A)\geq 0 \qquad (\alpha \in [0,1])
\end{eqnarray}
for every self-adjoint operator $A\in\mathcal{A}$ and each positive operator $\rho \in \mathcal{A}$.
\end{lemma}
\begin{proof}
Put $\Delta=\{\alpha \in [0,1] : \Phi(\rho A^2) \geq \Phi(\rho^{\alpha}A\rho^{1-\alpha} A)\}$.
Clearly $ \{0,1\} \subseteq \Delta$ and the set $\Delta$ is closed, since the map $\alpha \longrightarrow \Phi(\rho^{\alpha}A\rho^{1-\alpha} A)$ is norm continuous. Therefore, to prove $[0,1]\subseteq \Delta$ it is enough to show that $\alpha ,\beta \in \Delta$ implies $\frac{\alpha+\beta}{2}\in \Delta$. It follows from the tracial positivity of $\Phi$ that
\begin{eqnarray*}
0&\leq& \Phi\left((\rho^{\frac{1-\alpha}{2}} A\rho^{\frac{\alpha}{2}} -\rho^{\frac{1-\beta}{2}} A\rho^{\frac{\beta}{2}})^*(\rho^{\frac{1-\alpha}{2}} A\rho^{\frac{\alpha}{2}} -\rho^{\frac{1-\beta}{2}} A\rho^{\frac{\beta}{2}})\right)
\\ &=& \Phi\left((\rho^{\frac{\alpha}{2}}A\rho^{\frac{1-\alpha}{2}} -\rho^{\frac{\beta}{2}}A\rho^{\frac{1-\beta}{2}} )(\rho^{\frac{1-\alpha}{2}} A\rho^{\frac{\alpha}{2}} -\rho^{\frac{1-\beta}{2}} A\rho^{\frac{\beta}{2}})\right)
\\ &=& \Phi\left( \rho^{\frac{\alpha}{2}}A\rho^{1-\alpha}A \rho^{\frac{\alpha}{2}} -\rho^{\frac{\alpha}{2}}A\rho^{1-\frac{\alpha+\beta}{2}}A \rho^{\frac{\beta}{2}} -\rho^{\frac{\beta}{2}}A\rho^{1-\frac{\alpha+\beta}{2}}A \rho^{\frac{\alpha}{2}}+\rho^{\frac{\beta}{2}}A\rho^{1-\beta}A \rho^{\frac{\beta}{2}}\right)\\ &=& \Phi( \rho^{\alpha}A\rho^{1-\alpha}A ) -\Phi(\rho^{\frac{\alpha+\beta}{2}}A\rho^{1-\frac{\alpha+\beta}{2}}A) -\Phi(\rho^{\frac{\alpha+\beta}{2}}A\rho^{1-\frac{\alpha+\beta}{2}}A) +\Phi(\rho^{\beta}A\rho^{1-\beta}A).
\end{eqnarray*}
Hence,
\begin{eqnarray}\label{concave}
\Phi( \rho^{\alpha}A\rho^{1-\alpha}A ) +\Phi(\rho^{\beta}A\rho^{1-\beta}A )\geq 2\Phi(\rho^{\frac{\alpha+\beta}{2}}A\rho^{1-\frac{\alpha+\beta}{2}}A),
\end{eqnarray}
since $\alpha,\beta \in \Delta$, we infer
\begin{eqnarray*}
\Phi(\rho A^2)
&\geq& \Phi(\rho^{\frac{\alpha+\beta}{2}}A\rho^{1-\frac{\alpha+\beta}{2}}A),
\end{eqnarray*}
which implies that $\frac{\alpha+\beta}{2}\in \Delta$.
\end{proof}
\begin{remark}
Inequality (\ref{concave}) shows that the map $\alpha \rightarrow \Phi(\rho^{\alpha}A\rho^{1-\alpha}A)$ is convex. In addition, we have
\begin{eqnarray}\label{karan}
I_{\rho,\Phi}^{\alpha}(A) \leq I_{\rho,\Phi}(A) \ \ (\alpha \in \mathbb{R})
\end{eqnarray}
for every self-adjoint operator $A$ and each positive operator $\rho$. Indeed, using the positivity and the tracial property of $\Phi$, we get
\begin{eqnarray*}
0&\leq& \Phi\left((\rho^{\frac{\alpha}{2}} A\rho^{\frac{1-\alpha}{2}} -\rho^{\frac{1-\alpha}{2}} A\rho^{\frac{\alpha}{2}})(\rho^{\frac{1-\alpha}{2}} A\rho^{\frac{\alpha}{2}} -\rho^{\frac{\alpha}{2}} A\rho^{\frac{\alpha}{2}})\right)
\\ &=& 2\Phi(\rho^{\alpha}A\rho^{1-\alpha}A) -2\Phi(\rho^{\frac{1}{2}}A\rho^{\frac{1}{2}}A).
\end{eqnarray*}
Therefore $\Phi(\rho^{\alpha}A\rho^{1-\alpha}A) \geq \Phi(\rho^{\frac{1}{2}}A\rho^{\frac{1}{2}}A)$ and so
\begin{eqnarray*}
I_{\rho,\Phi}^{\alpha}(A)= \Phi(\rho A^2)-\Phi(\rho^{\alpha}A\rho^{1-\alpha}A) \leq \Phi(\rho A^2)-\Phi(\rho^{\frac{1}{2}}A\rho^{\frac{1}{2}}A)=I_{\rho,\Phi}(A).
\end{eqnarray*}
\end{remark}

Let $\Phi$ be a tracial positive linear map between $C^*$ algebras. It follows from Lemma \ref{skew} that $ I_{\rho,\Phi}^{\alpha}(A)\geq0$ for every self-adjoint operator $A$, but it is not true in general when $A$ is an arbitrary operator. Hence even if $\Phi$ is a tracial positive functional, then $\mbox{Corr}_{\rho,\Phi}^{\alpha}(\cdot,\cdot)$ cannot induce a complex valued semi-inner product and we cannot use the Cauchy--Schwarz inequality; see \cite[Remark IV.2]{yangi}. The next lemma helps us to give a positive definite version of the generalized correlation.
\begin{lemma}\label{corrmosbat}
Let $\Phi: \mathcal{A}\longrightarrow \mathcal{B}$ be a tracial positive linear map and $\rho\in \mathcal{A}$ be a positive operator. Then
\begin{eqnarray*}
I_{\rho,\Phi}^{\alpha}(A) +I_{\rho,\Phi}^{\alpha}(A^*)\geq 0
\end{eqnarray*}
for every $A \in \mathcal{A}$.
\end{lemma}
\begin{proof}
We define $\tilde{\Phi} :M_2(\mathcal{A}) \longrightarrow B$ by $\tilde{\Phi} \left(\begin{bmatrix} A&B\\ C&D \end{bmatrix}\right)=\dfrac{1}{2}\Phi(A+D)$. It is obvious that $\tilde{\Phi}$ is a tracial positive linear map. Take $ \tilde{A}=\begin{bmatrix} 0&A^*\\ A&0 \end{bmatrix}$ and $ \tilde{\rho}=\begin{bmatrix} \rho & 0\\ 0&\rho \end{bmatrix}$. Clearly $\tilde{A}$ is a self-adjoint operator in $M_2(\mathcal{A})$ and $\tilde{\rho}$ is a $\tilde{\Phi}$-density operator. Using Lemma \ref{skew} for $\tilde{\Phi}$, we get
\begin{eqnarray*}
\frac{1}{2}\left( \Phi(\rho A^*A)+\Phi(\rho AA^*)\right) &=& \tilde{\Phi}\left(\begin{bmatrix} \rho & 0\\ 0&\rho \end{bmatrix} \begin{bmatrix} 0&A^*\\ A&0 \end{bmatrix}\begin{bmatrix} 0&A^*\\ A&0 \end{bmatrix}\right)\\&=& \tilde{\Phi} (\tilde{\rho} \tilde{A}^2)\\ &\geq& \tilde{\Phi} (\tilde{\rho}^{1-\alpha} \tilde{A}\tilde{\rho}^\alpha \tilde{A})\\ && \qquad \qquad \qquad ({\rm by \ Lemma \ \ref{skew}}) \\ &=& \tilde{\Phi}\left(\begin{bmatrix} \rho^{1-\alpha} & 0\\ 0&\rho^{1-\alpha} \end{bmatrix} \begin{bmatrix} 0&A^*\\ A&0 \end{bmatrix}\begin{bmatrix} \rho^\alpha & 0\\ 0&\rho^\alpha \end{bmatrix}\begin{bmatrix} 0&A^*\\ A&0 \end{bmatrix}\right)
\\ &=&\frac{1}{2} (\Phi(\rho^{1-\alpha} A^*\rho^\alpha A)+\Phi(\rho^{1-\alpha} A\rho^\alpha A^*)).
\end{eqnarray*}
Hence,
\begin{eqnarray*}
I_{\rho,\Phi}^{\alpha}(A) +I_{\rho,\Phi}^{\alpha}(A^*) &=& \Phi(\rho A^*A)+\Phi(\rho AA^*)\\ && - \Phi(\rho^{1-\alpha} A^*\rho^\alpha A)-\Phi(\rho^{1-\alpha} A\rho^\alpha A^*) \geq 0.
\end{eqnarray*}
\end{proof}
\begin{definition}\label{semi}
Let $\Phi: \mathcal{A}\longrightarrow \mathcal{B}$ be a tracial positive linear map and $\rho\in \mathcal{A}$ be a $\Phi$-density operator. Then for every operators $A,B \in \mathcal{A}$, we set
\begin{eqnarray*}
{\rm Corr}_{\rho,\Phi}'^{\alpha}(A,B):=\dfrac{1}{2} \left({\rm Corr}_{\rho,\Phi}^{\alpha}(A,B)+{\rm Corr}_{\rho,\Phi}^{\alpha}(B^*,A^*)\right)\ {\rm and} \ I_{\rho,\Phi}'^{\alpha}(A):={\rm Corr}_{\rho,\Phi}'^{\alpha}(A,A).
\end{eqnarray*}

It is easy to check that ${\rm Corr}_{\rho,\Phi}'^{\alpha}(A,B)$ has the following properties:
\begin{itemize}
\item [(i)] ${\rm Corr}_{\rho,\Phi}'^{\alpha}(A,A) \geq 0, {\rm \ for\ every}\ A\in \mathcal{A}, \ {\rm (by \ Lemma\ \ref{corrmosbat})}$,
\item [(ii)] ${\rm Corr}_{\rho,\Phi}'^{\alpha}(A,B+\lambda C) ={\rm Corr}_{\rho,\Phi}'^{\alpha}(A,B) +\lambda {\rm Corr}_{\rho,\Phi}'^{\alpha}(A,C), {\rm \ for\ all}\ A,B\in \mathcal{A} {\rm \ and\ every\ } \lambda \in \mathbb{C}$,
\item [(iii)] ${\rm Corr}_{\rho,\Phi}'^{\alpha}(A,B)^*={\rm Corr}_{\rho,\Phi}'^{\alpha}(B,A)$.
\end{itemize}
\end{definition}
Next we give a generalization of inequality (\ref{semic}).

\begin{theorem}
Let $\mathcal{A}$ be a $C^*$-algebra and $\mathcal{B}$ be $C^*$-subalgebra of $\mathcal{A}$. If $\mathcal{E}:\mathcal{A}\longrightarrow \mathcal{B}$ is a tracial conditional expectation, then
\begin{eqnarray*}
\left|{\rm Re} ({\rm Corr}_{\rho,\mathcal{E}}^{\alpha}(A,B))\right|^2 \leq I_{\rho,\mathcal{E}}^{\alpha}(A) I_{\rho,\mathcal{E}}^{\alpha}(B)
\end{eqnarray*}
for all self-adjoint operators $A,B \in \mathcal{A}$ and each $\mathcal{E}$-density operator $\rho \in \mathcal{A}$.
\end{theorem}
\begin{proof}
Define the map $\langle\cdot , \cdot \rangle :\mathcal{A} \times \mathcal{A} \longrightarrow \mathcal{B}$ by $\langle A ,B \rangle={\rm Corr}_{\rho,\mathcal{E}}'^{\alpha}(A,B)$. If $A,B \in \mathcal{A}$ and $C\in \mathcal{B}$, then
\begin{eqnarray*}
\langle A ,BC \rangle &=& {\rm Corr}_{\rho,\mathcal{E}}'^{\alpha}(A,BC)\\ &=&\frac{1}{2} \left({\rm Corr}_{\rho,\mathcal{E}}^{\alpha}(A,BC)+{\rm Corr}_{\rho,\mathcal{E}}^{\alpha}(C^*B^*,A^*)\right)\\ &=& \frac{1}{2}\left(\mathcal{E}(\rho A^*BC)-\mathcal{E}(\rho^{1-\alpha}A^* \rho^\alpha BC) +\mathcal{E}(\rho BCA^*)-\mathcal{E}(\rho^{1-\alpha}BC \rho^\alpha A^*)\right)\\ &=& \frac{1}{2}\left(\mathcal{E}(\rho A^*BC)-\mathcal{E}(\rho^{1-\alpha}A^* \rho^\alpha BC) +\mathcal{E}(CA^*\rho B)-\mathcal{E}(C \rho^{\alpha} A^*\rho^{1-\alpha} B)\right) \\ && \qquad \qquad \qquad \qquad \qquad \qquad \qquad \qquad \qquad ({\rm since \ \mathcal{E} \ is \ tracial})\\ &=& \frac{1}{2}\left(\mathcal{E}(\rho A^*B)C-\mathcal{E}(\rho^{1-\alpha}A^* \rho^\alpha B)C +\mathcal{E}(\rho BA^*)C-\mathcal{E}(\rho^{1-\alpha} B\rho^\alpha A^*)C\right) \\ &&\ \ \ \ ({\rm since \ \mathcal{E} \ is\ a \ conditional \ expectation \ and \ by\ equality \ (\ref{concom})})\\ &=& {\rm Corr}_{\rho,\mathcal{E}}'^{\alpha}(A,B)C \\ &=& \langle A ,B \rangle C.
\end{eqnarray*}
Using this fact and Definition \ref{semi} we see that $(\mathcal{A},\langle\cdot , \cdot \rangle)$ is a semi-inner product $\mathcal{B}$-module. Moreover, equality (\ref{concom}) shows that ${\rm ran}(Z)\subseteq \mathcal{Z}(\mathcal{B})$. If $A$ and $B$ are self-adjoint operators in $\mathcal{A}$, then we get
\begin{eqnarray*}
\left|{\rm Re} ({\rm Corr}_{\rho,\mathcal{E}}^{\alpha}(A,B))\right|^2&=&\left|\dfrac{1}{2}\left( {\rm Corr}_{\rho,\mathcal{E}}^{\alpha}(A,B)+{\rm Corr}_{\rho,\mathcal{E}}^{\alpha}(B,A)\right)\right|^2 \\ &=& \left| {\rm Corr}_{\rho,\mathcal{E}}'^{\alpha}(A,B)\right|^2\\ &=& \left| \langle A ,B \rangle \right|^2 \\ &\leq& \langle A ,A \rangle \langle B,B\rangle \qquad\ ({\rm by\ inequality}\ (\ref{cuachysharp})) \\ &=& I_{\rho,\mathcal{E}}^{\alpha}(A) I_{\rho,\mathcal{E}}^{\alpha}(B) \qquad ({\rm since}\ A,B \ {\rm are \ }{\rm self-adjoint}).
\end{eqnarray*}
\end{proof}

Let $A$ be a self-adjoint operator and $\rho$ be a density operator. According to \cite[Section III]{luo2} we have $I_{\rho}(A)\leq V_{\rho}(A)$. We give a generalization of this inequality for a tracial 2-positive linear map $\Phi$ and a $\Phi$-density operator $\rho$. It follows from Lemma \ref{matrixp} that the matrix
\begin{eqnarray*}
\begin{bmatrix} \rho^\frac{1}{4} A \rho^\frac{1}{4} \rho^\frac{1}{4}A\rho^\frac{1}{4} & \rho^\frac{1}{4}A\rho^\frac{1}{4} \rho^\frac{1}{2}\\ \rho^\frac{1}{2}\rho^\frac{1}{4}A \rho^\frac{1}{4} & \rho \end{bmatrix}
\end{eqnarray*}
is positive. Since $\Phi$ is 2-positive, we have
\begin{eqnarray*}
\begin{bmatrix} \Phi(\rho^\frac{1}{4} A \rho^\frac{1}{2} \rho^\frac{1}{2}A\rho^\frac{1}{4}) & \Phi(\rho^\frac{1}{4}A\rho^\frac{1}{4} \rho^\frac{1}{2})\\ \Phi(\rho^\frac{1}{2}\rho^\frac{1}{4}A \rho^\frac{1}{4}) & \Phi(\rho) \end{bmatrix}\geq 0.
\end{eqnarray*}
Therefore, by using Lemma \ref{matrixp} and applying the tracial property of $\Phi$ we get
\begin{eqnarray}
\Phi(\rho^\frac{1}{2}A \rho^\frac{1}{2}A)\geq \Phi(\rho A) \Phi(\rho)^{-1} \Phi(\rho A) =\Phi(\rho A) ^2,
\end{eqnarray}
which implies that $I_{\rho,\Phi}(A)\leq V_{\rho,\Phi}(A)$. Consequently, by using inequality (\ref{karan}), we reach $I_{\rho,\Phi}^{\alpha} (A)\leq V_{\rho,\Phi}(A)$.\\
For a tracial positive linear map $\Phi$ and a self-adjoint operator $A$, we set
\begin{eqnarray*}
 J_{\rho,\Phi}(A):=2 V_{\rho,\Phi}(A)-I_{\rho,\Phi}(A)\ {\rm and}\ U_{\rho,\Phi}(A):= I_{\rho,\Phi}(A) \sharp J_{\rho,\Phi}(A).
\end{eqnarray*}
 Since $U_{\rho,\Phi}(A)\leq V_{\rho,\Phi}(A)$ (by the arithmetic-geometric mean inequality), the next theorem is a refinement of Theorem \ref{main1} in the case that $\Phi$ is a conditional expectation. To establish it, we model the classical techniques (see \cite{luo}) to the non commutative framework.

\begin{theorem}\label{maincon}
Let $\mathcal{A}$ be a $C^*$-algebra and $\mathcal{B}$ be a $C^*$-subalgebra of $\mathcal{A}$. If $\mathcal{E}:\mathcal{A}\longrightarrow \mathcal{B}$ is a tracial conditional expectation, then
\begin{eqnarray}
U_{\rho,\mathcal{E}}(A) U_{\rho,\mathcal{E}}(B) \geq \frac{1}{4}|\mathcal{E}(\rho [A,B])|^2
\end{eqnarray}
for all self-adjoint operators $A,B \in \mathcal{A}$ and each $\mathcal{E}$-density operator $\rho \in \mathcal{A}$.
\end{theorem}
\begin{proof}
Consider self-adjoint operators $A_0=A-\mathcal{E}(\rho A)$ and $B_0=B-\mathcal{E}(\rho B)$. A simple calculation shows that
\begin{eqnarray}\label{IJ}
I_{\rho,\mathcal{E}}(A)=\frac{1}{2} \mathcal{E}((\mbox{i}[\rho^\frac{1}{2},A_0])^2) \ {\rm and}\ J_{\rho,\mathcal{E}}(B)=\frac{1}{2} \mathcal{E} (\{\rho^\frac{1}{2},B_0\}^2),
\end{eqnarray}
where $\{\rho^\frac{1}{2},B\}=\rho^\frac{1}{2}B_0+B_0\rho^\frac{1}{2}$. Indeed,
\begin{eqnarray*}
\mathcal{E}((\mbox{i}[\rho^\frac{1}{2},A_0])^2)&=&-\mathcal{E}(\rho^\frac{1}{2}A_0\rho^\frac{1}{2}A_0-\rho^\frac{1}{2}A_0^{2}\rho^\frac{1}{2}-A_0\rho^\frac{1}{2}\rho^\frac{1}{2}A_0+A_0\rho^\frac{1}{2}A_0\rho^\frac{1}{2})\\ &=& 2 \mathcal{E}(\rho A_0^2)-2\mathcal{E}(\rho^\frac{1}{2}A_0\rho^\frac{1}{2}A_0)\\ && \qquad \qquad \qquad ({\rm by\
 the \ tracial\ property\ of }\ \mathcal{E})\\ &=& 2\mathcal{E}\Big(\rho (A-\mathcal{E}(\rho A))^2\Big)-2\mathcal{E}\Big(\rho^\frac{1}{2} (A-\mathcal{E}(\rho A))\rho^\frac{1}{2}(A-\mathcal{E}(\rho A))\Big)\\ &=& 2\mathcal{E}(\rho A^2)-2(\mathcal{E}(\rho A))^2 - 2\mathcal{E}(\rho^\frac{1}{2}A\rho^\frac{1}{2}A)+2(\mathcal{E}(\rho A))^2\\ && \qquad \qquad \qquad ({\rm since\ \mathcal{E}\ is\ a \ conditional\ expectation}) \\ &=&2\mathcal{E}(\rho A^2) -2 \mathcal{E}(\rho^\frac{1}{2}A\rho^\frac{1}{2}A).
\end{eqnarray*}
Similarly, we can establish the other inequality in (\ref{IJ}).\\
Let $Z\in \mathcal{B}$. Then
\begin{eqnarray*}
\mathcal{E}\Big(\mbox{i}Z [\rho^\frac{1}{2},A_0] \{\rho^\frac{1}{2},B_0\} &+& \mbox{i} \{\rho^\frac{1}{2},B_0\}[\rho^\frac{1}{2},A_0]Z\Big)\\ &=&Z\mathcal{E}\Big(\mbox{i}\big(\rho^\frac{1}{2}A_0\rho^\frac{1}{2}B_0 +\rho^\frac{1}{2}A_0B_0\rho^\frac{1}{2}-A_0\rho^\frac{1}{2}\rho^\frac{1}{2}B_0\\ && -A_0\rho^\frac{1}{2}B_0\rho^\frac{1}{2}+\rho^\frac{1}{2}B_0\rho^\frac{1}{2}A_0\\ && -\rho^\frac{1}{2}B_0A_0\rho^\frac{1}{2}+B_0\rho^\frac{1}{2}\rho^\frac{1}{2}A_0-B_0\rho^\frac{1}{2}A_0\rho^\frac{1}{2}\big)\Big)\\ && \qquad \qquad \qquad ({\rm by\ equality\ (\ref{concom}}))\\ &=&2\mbox{i}Z\mathcal{E}(\rho[A_0,B_0])\\ &=& 2\mbox{i}Z\mathcal{E} \big(\rho \left[A-\mathcal{E}(\rho A), B-\mathcal{E}(\rho B)\right]\big)\\ &=& 2\mbox{i}Z\mathcal{E}\Big( \rho \big((A-\mathcal{E}(\rho A))(B-\mathcal{E}(\rho B))\\ && - (B-\mathcal{E}(\rho B)) (A-\mathcal{E}(\rho A))\big) \Big)\\ &=&2\mbox{i}Z\mathcal{E} \Big(\rho \big(AB-A\mathcal{E}(\rho B)-\mathcal{E}(\rho A)B+\mathcal{E}(\rho A)\mathcal{E}(\rho B)\\ &&-BA+B\mathcal{E}(\rho A)+\mathcal{E}(\rho B)A-\mathcal{E}(\rho B)\mathcal{E}(\rho A)\big)\Big) \\ &=&2\mbox{i}Z\big(\mathcal{E} (\rho AB)- \mathcal{E}(\rho A)\mathcal{E}(\rho B)-\mathcal{E}(\rho A)\mathcal{E}(\rho B)\\ && +\mathcal{E}(\rho) \mathcal{E}(\rho A)\mathcal{E}(\rho B)-\mathcal{E}(\rho BA)+\mathcal{E}(\rho B)\mathcal{E}(\rho A)\\ && +\mathcal{E}(\rho B)\mathcal{E}(\rho A)-\mathcal{E}(\rho)\mathcal{E}(\rho B)\mathcal{E}(\rho A)\big) \\ &=&2\mbox{i}Z\mathcal{E}(\rho[A,B])\\ && \qquad \qquad \qquad ({\rm by \ equality \ \ref{concom} }).
\end{eqnarray*}
Hence,
\begin{eqnarray}\label{rabete1}
\mathcal{E}\Big(\mbox{i}Z [\rho^\frac{1}{2},A_0] \{\rho^\frac{1}{2},B_0\} &+& \mbox{i} \{\rho^\frac{1}{2},B_0\}[\rho^\frac{1}{2},A_0]Z\Big)=2\mbox{i}Z\mathcal{E}(\rho[A,B]).
\end{eqnarray}
Let $Z \in \mathcal{B}$ be a self-adjoint operator and $X=\mbox{i}[\rho^\frac{1}{2},A_0]Z +\{\rho^\frac{1}{2},B_0\}$. Then
\begin{eqnarray*}
0\leq \mathcal{E}(X^*X)&=& \mathcal{E}\Big(\big(\mbox{i}Z [\rho^\frac{1}{2},A_0]+\{\rho^\frac{1}{2},B_0\}\big)\big(\mbox{i}[\rho^\frac{1}{2},A_0]Z+\{\rho^\frac{1}{2},B_0\}\big)\Big) \\ && \qquad \qquad \qquad \qquad \qquad ({\rm because \ \mbox{i}[\rho^\frac{1}{2},A_0]\ is \ }{\rm self-adjoint}) \\ &=& \mathcal{E}\Big(-Z [\rho^\frac{1}{2},A_0] [\rho^\frac{1}{2},A_0]Z+iZ [\rho^\frac{1}{2},A_0] \{\rho^\frac{1}{2},B_0\} \\ &&+ \mbox{i} \{\rho^\frac{1}{2},B_0\}[\rho^\frac{1}{2},A_0]Z +\{\rho^\frac{1}{2},B_0\}^2\Big)\\ &=&2I_{\rho,\mathcal{E}}(A)Z^2+2\mbox{i} \mathcal{E}(\rho [A,B]) Z+2J_{\rho,\mathcal{E}}(B) \\ && \qquad \qquad \qquad \qquad \qquad ({\rm by \ equality\ (\ref{IJ})\ and \ equality\ (\ref{rabete1}) }).
\end{eqnarray*}
Without loss of the generality we can assume that $I_{\rho,\mathcal{E}}(A)>0$. If we put $Z:=- \frac{\mbox{i}}{2} I_{\rho,\mathcal{E}}(A)^{-1} \mathcal{E}(\rho [A,B])$, then we get
\begin{eqnarray*}
-\frac{1}{4}I_{\rho,\mathcal{E}}(A)^{-1}\mathcal{E}(\rho[A,B])^2+\frac{1}{2}I_{\rho,\mathcal{E}}(A)^{-1}\mathcal{E}(\rho[A,B])^2+J_{\rho,\mathcal{E}}(B) \geq 0.
\end{eqnarray*}
or equivalently,
\begin{eqnarray}\label{ij1}
I_{\rho,\mathcal{E}}(A)J_{\rho,\mathcal{E}}(B) \geq -\frac{1}{4}\mathcal{E}(\rho[A,B])^2=\frac{1}{4}|\mathcal{E}(\rho [A,B])|^2,
\end{eqnarray}
since $ \mathcal{E}(\rho[A,B])^*=-\mathcal{E}(\rho[A,B])$.
It follows from the fact that for every $X\in \mathcal{A}$, $\mathcal{E}(X)\subseteq \mathcal{Z}(\mathcal{B})$ (equality (\ref{concom})), we have
\begin{eqnarray*}
U_{\rho,\mathcal{E}}(A) U_{\rho,\mathcal{E}}(B)&=& \big(I_{\rho,\mathcal{E}}(A)\sharp J_{\rho,\mathcal{E}}(A)\big) \big(I_{\rho,\mathcal{E}}(B)\sharp J_{\rho,\mathcal{E}}(B)\big) \\ &=& \big(I_{\rho,\mathcal{E}}(A) J_{\rho,\mathcal{E}}(B)\big)^\frac{1}{2} \big(I_{\rho,\mathcal{E}}(B) J_{\rho,\mathcal{E}}(A)\big)^\frac{1}{2}\\ && \qquad ({\rm by \ the \ commutivity\ property\ in\ equality\ (\ref{concom}) })\\ &\geq& \frac{1}{4}|\mathcal{E}(\rho [A,B])|^2\\ && \qquad ({\rm by \ inequlity\ (\ref{ij1} )}).
\end{eqnarray*}
\end{proof}
As a consequence we get the following result of Luo \cite{luo}.
\begin{corollary}\cite[p. 2]{luo}
If $A,B$ are two self-adjoint operators, then
\begin{eqnarray*}
U_{\rho}(A)U_{\rho}(B)\geq \frac{1}{4} |{\rm Tr}(\rho [A,B])|^2.
\end{eqnarray*}

\end{corollary}

\end{document}